\documentclass[10pt,reqno]{amsart}
\usepackage{amsmath} 
\usepackage{amssymb}
\usepackage{dsfont}
\usepackage[dvips,draft,final]{graphics}
\usepackage[T1]{fontenc}
\usepackage{lmodern}
\usepackage{fancyhdr}
\usepackage{url}
\usepackage{enumerate}
\usepackage{color}
\usepackage{graphicx}
\usepackage{mathrsfs}

\newtheorem{theorem}{Theorem}[section]
\newtheorem{proposition}[theorem]{Proposition}
\newtheorem{lemma}[theorem]{Lemma}
\newtheorem{follow}[theorem]{Corollary}

\theoremstyle{definition}
\newtheorem{remark}[theorem]{Remark}

\newcommand{\bel}{\begin{equation} \label}
\newcommand{\ee}{\end{equation}}

\newcommand{\pd}{\partial}

\newcommand{\R}{{\mathbb R}}
\newcommand{\N}{{\mathbb N}}

\newcommand{\re}{\mathfrak R}
\newcommand{\im}{\mathfrak I}

\newcommand{\sH}{{\mathscr H}}

\def\epsilon{\varepsilon}
\def\phi {\varphi}

\def\beq{\begin{equation}}
\def\eeq{\end{equation}}
\renewcommand{\leq}{\leqslant}
\renewcommand{\geq}{\geqslant}
\newcommand{\bea}{\begin{eqnarray}}
\newcommand{\eea}{\end{eqnarray}}
\newcommand{\beas}{\begin{eqnarray*}}
\newcommand{\eeas}{\end{eqnarray*}}

{

\providecommand{\abs}[1]{\left\lvert#1\right\rvert}
\providecommand{\norm}[1]{\left\lVert#1\right\rVert}

\title[Recovery of non compactly supported coefficients of  elliptic equations]{Recovery of non compactly supported coefficients of  elliptic equations on an infinite waveguide}
\author{Yavar Kian}
\address{Aix Marseille Univ, Universit\'e de Toulon, CNRS, CPT, Marseille, France.}
\email{yavar.kian@univ-amu.fr}

\begin{document}
\begin{abstract}
We consider the unique recovery of a non compactly supported and non periodic perturbation of a Schr\"odinger operator in an unbounded cylindrical domain, also called waveguide, from boundary measurements. More precisely, we prove recovery of general class of electric potentials from  the partial Dirichlet-to-Neumann map, where the Dirichlet data  is supported on slightly more than half of the boundary and the Neumann data is taken on the other half of the boundary. We apply this result in different context including recovery of some general class of coefficients from measurements on a bounded subset and recovery of an electric potential, supported on an unbounded cylinder,   of a Schr\"odinger operator in a slab.


\end{abstract}
\maketitle


\section{Introduction}
\label{sec-intro}
\setcounter{equation}{0}
Let $\Omega :=  \omega\times \R $, where 
$\omega$ is a bounded open set of $\mathbb{R}^2$, with $C^2$-boundary. Throughout this paper we denote the  point $x \in \Omega$ by $x=(x',x_3)$, where $x_3 \in \R$ and $x':= (x_1,x_2) \in \omega$. Given $q \in L^\infty(\Omega)$ such that $0$ is not in the spectrum of $-\Delta+q$  with Dirichlet boundary condition, 
we consider the following boundary value problem (BVP in short):
\bel{eq1}
\left\{ 
\begin{array}{rcll} 
(-\Delta + q) v & = & 0, & \mbox{in}\ \Omega,\\ 
v & = & f ,& \mbox{on}\ \Gamma : = \pd \Omega.
\end{array}
\right.
\ee
Since $\Gamma=  \pd \omega\times\R $, the outward unit vector $\nu$ normal to $\Gamma$ reads
$$ \nu(x',x_3)=(\nu'(x'),0),\ x=(x',x_3)\in\Gamma, $$
where $\nu'$ is the outer unit normal vector of $\pd \omega$.
For  simplicity, we  refer to $\nu$ for both exterior unit vectors normal to $\pd  \omega$ and to $\Gamma$.
For $\theta_0 \in \mathbb{S}^1 :=\{ y\in\R^2;\ \abs{y}=1\}$ fixed, we introduce the $\theta_0$-illuminated (resp., $\theta_0$-shadowed) face of $\pd \omega$, as
\bel{xi-isf} 
\pd \omega_{\theta_0}^- := \{ x \in \pd \omega;\ \theta_0 \cdot \nu(x) \leq 0 \}\ (\mbox{resp.},\ \pd \omega_{\theta_0}^+= \{x \in \pd \omega;\ \theta_0 \cdot \nu(x) \geq 0\}).
\ee
Here and in the remaining part of this text, we denote by 
$x \cdot y := \sum_{j=1}^k x_j y_j$
the Euclidian scalar product of any two vectors $x:=(x_1,\ldots,x_k)$ and $y:=(y_1,\ldots,y_k)$ of $\R^k$, for $k \in \N^*$, and we put $|x|:=(x \cdot x)^{1 \slash 2}$.

Set $G:=G'\times \R$, where $G'$ 
is an arbitrary open set of $\partial\omega$
containing the compact set $\pd \omega_{\theta_0}^-$ in $\pd \omega$ and consider $K=K'\times \R$ with $K'$ an arbitrary open set
containing the compact set $\pd \omega_{\theta_0}^+$ in $\pd \omega$. In the present paper we seek determination of $q$ from the knowledge of the partial Dirichlet-to-Neumann (DN in short) map
\bel{a1}
\Lambda_{q} :  f  \mapsto {\pd_\nu v}_{|G},
\ee
where $\pd_\nu v (x) := \nabla v (x) \cdot \nu(x)$ is the normal derivative of the solution $v$ to \eqref{eq1}, computed at $x \in \Gamma$ and supp$(f)\subset K$.
\subsection{Physical motivations}
Let us recall that the problem under consideration in this paper is related to the so called  electrical impedance tomography (EIT in short) and its several applications in medical imaging  and others.
Note that the specific geometry of infinite cylinder or closed waveguide can be considered for problems of transmission to long distance or transmission through particular structures, where the ratio length-to-diameter is really high, such as nanostructures. In this context, the problem addressed in this paper can correspond to the unique recovery of an  impurity  perturbing the guided propagation (see  \cite{CL,KBF}). Let us also observe that in Corollary \ref{c3}, we show how one can apply our result to  the problem stated in a slab, which is frequently used  for modeling  propagation in shallow-ocean acoustics (e.g. \cite{AK}), for coefficients supported in an infinite cylinder. 
 
\subsection{Known results}
Since the pioneer work of \cite{Ca}, the celebrated Calder\'on or the EIT 
problem  has been growing in interest. In \cite{SU}, Sylvester and Uhlmann provide one of the first and most important results related to   this problem. They actually proved, in dimension $n \geq 3$, the unique recovery of a smooth conductivity from the full DN map. Since then, several authors extended this result in several way.  The determination of an unknown coefficient from partial knowledge of the DN map was first addressed in \cite{BU} and extended by  Kenig, Sj\"ostrand and Uhlmann  in \cite{KSU}  to the recovery of a potential from restriction of data to  the back and the front face illuminated by a point lying outside  the convex hull of the domain. In dimension two, similar results  with full and partial data have been stated in  \cite{B,IUY1, IUY2}. We mention also, without being exhaustive, the work of \cite{CDR1,CDR2,CKS1,Pot1,Pot2} dealing with the stability issue associated to this problem and some results inspired by this approach for other PDEs stated in \cite{CK,HK,Ki2,Ki3,KO}.

Let us remark that all the above mentioned results have been proved in a bounded domain. It appears that  only a small number of mathematical papers deal with inverse boundary value problems in an unbounded domain. Combining results of unique continuation with  complex geometric optics (CGO in short) solutions and a Carleman estimate borrowed from \cite{BU},   Li and Uhlmann proved in \cite{LU} the unique recovery of compactly supported electric potentials of the stationary Schr\"odinger operator in a slab  from  partial boundary measurements. In \cite{KLU}, the authors extended this result to magnetic Schr\"odinger operators and \cite{CM} treated the stability issue for this inverse problem. We mention also \cite{L1,L2} dealing with more general Schr\"odinger equations, the work of  \cite{Y} for bi-harmonic operators and the recovery of an embedded object in a slab treated by \cite{Ik,SW}. More recently, \cite{CKS2,CKS3} proved the stable recovery of coefficients periodic along the axis of an infinite cylindrical domain. Finally, we mention \cite{BKS,BKS1, CS, KKS,Ki1,  KPS1, KPS2} dealing with determination of non-compactly supported coefficients appearing in different PDEs from boundary measurements.

\subsection{Statement of the main result and applications}

Prior to stating the main result of this article we first recall some results stated in \cite{BU,CKS2,CKS3} related to the well-posedness of the BVP \eqref{eq1} in the space $H_\Delta(\Omega):=\{ u \in L^2(\Omega);\ \Delta u \in L^2(\Omega)\}$ with the norm
$$ \norm{u}^2_{H_\Delta(\Omega)} :=\norm{u}_{L^2(\Omega)}^2+\norm{\Delta u}_{L^2(\Omega)}^2. $$
Since $\Omega$ is unbounded, for $X=\omega$ or $X=\partial\omega$ and any $s>0$, we define the space $H^s(X\times\R)$ by
$$H^s(X\times \R):=L^2(X;H^s(\R))\cap L^2(\R;H^s(X)).$$
We define also $H^{-s}(\Gamma)$ to be the dual space of $H^s(\Gamma)$.
Combining \cite[Lemma 1.1]{BU} with \cite[Lemma 2.2]{CKS2}, we deduce that the map
$$ \mathcal T_0 u :=u_{\vert\Gamma}\ (\mbox{resp.,}\ \mathcal T_1 u :={\pd_\nu u}_{\vert\Gamma}),\ u \in C_0^\infty(\R^3), $$
extend into a continuous function $\mathcal T_0 : H_\Delta(\Omega) \to H^{-\frac{1}{2}}(\Gamma)$ (resp., $\mathcal T_1 : H_\Delta(\Omega) \to H^{-\frac{3}{2}}(\Gamma)$). We set the space 
$$\mathcal H(\Gamma):= \mathcal T_0 H_\Delta(\Omega) = \{ \mathcal T_0 u;\ u \in H_\Delta(\Omega) \},$$ 
and notice from \cite[Lemma 2.2]{CKS2} that $\mathcal T_0$ is bijective from $B:=\{ u \in L^2(\Omega);\ \Delta u = 0 \}$ onto $\sH(\Gamma)$. Thus, with reference to \cite{BU,NS}, we consider 
\bel{nhg}
\norm{f}_{\mathcal H(\Gamma)} : =\norm{\mathcal T_0^{-1} f}_{H_\Delta(\Omega)} = \norm{\mathcal T_0^{-1} f}_{L^2(\Omega)}.
\ee
We define also $\mathcal H_K(\Gamma):=\{f\in \mathcal H(\Gamma):\ \textrm{supp}(f)\subset K\}$.
Then, in view of \cite[Proposition 1.1]{CKS2}, assuming that $0$ is not in the spectrum of $-\Delta+q$ with Dirichlet boundary condition on $\Omega$, for any $f\in\sH(\Gamma)$ we deduce that the BVP \eqref{eq1} admits a unique solution $v\in L^2(\Omega)$. Moreover, the DN map $\Lambda_{q} : f \mapsto \mathcal T_1 v_{| G}$ is a bounded operator from $\sH_K(\Gamma)$ into $H^{-\frac{3}{2}}(G)$.

The main result of this paper can be stated as follows.

\begin{theorem}
\label{t1} 
Let  $q_1,q_2\in L^\infty(\Omega)$ be such that $q_1-q_2\in L^1(\Omega)$ and $0$ is not in the spectrum of $-\Delta+q_j$, $j=1,2$, with Dirichlet boundary condition on $\Omega$. Then the condition
\bel{t1a} 
\Lambda_{q_1}=\Lambda_{q_2}
\ee
implies $q_1=q_2$.
\end{theorem}

From the main result of this paper, stated in Theorem \ref{t1}, we deduce three other results related to other problems stated in an unbounded domain. The first application that we consider corresponds to the  Calder\'on problem stated in the unbounded domain $\Omega$. In order to state this problem, for $a_*\in(0,+\infty)$ and $a_0\in W^{2,\infty}(\Omega)$ satisfying $a_0\geq a_*$, we introduce the set of functions
$$\mathcal A:=\{a\in \mathcal C^1(\overline{\Omega})\cap H^2_{loc}(\Omega):\ a\geq a_*,\ \Delta\left(a^{\frac{1}{2}}\right)-\Delta\left(a_0^{\frac{1}{2}}\right)\in L^1(\Omega)\cap L^\infty(\Omega)\}$$
and, for $a\in \mathcal A$, the BVP 
\bel{a-eq1}
\left\{
\begin{array}{rcll} 
-\mbox{div}( a \nabla u ) & = & 0, & \mbox{in}\ \Omega,\\ 
u & = & f, & \mbox{on}\ \Gamma.
\end{array}
\right.
\ee
Recall that for any $a \in \mathcal A$ and any $f\in H^{\frac{1}{2}}(\Gamma)$, the BVP \eqref{a-eq1} admits a unique solution $u \in H^1(\Omega)$ for each $f \in H^{\frac{1}{2}}(\Gamma)$.
Moreover, the full DN map associated with \eqref{a-eq1}, defined by $f \mapsto a \mathcal T_1 u$ 
is a bounded operator from $H^{\frac{1}{2}}(\Gamma)$ to $H^{-\frac{1}{2}}(\Gamma)$. Here, we rather consider the partial DN map,
\bel{ca2} 
\Sigma_a : f \in H^{\frac{1}{2}}(\Gamma) \cap a^{-\frac{1}{2}}(\mathcal H_K(\Gamma)) \mapsto a \mathcal T_1 u_{\vert G},
\ee
where $a^{-\frac{1}{2}}(\mathcal H_K(\Gamma)):= \{a^{-\frac{1}{2}}f;\ f \in \mathcal H_K(\Gamma) \}$.
The first application of Theorem \ref{t1}  claims unique recovery of a conductivity  $a\in\mathcal A$, from the knowledge of $\Sigma_a$. It is stated as follows.

\begin{follow}
\label{c1} 
Let $\omega$ be connected and pick $a_j \in \mathcal A$, for $j=1,2$, obeying
\bel{ca3} 
a_1(x)=a_2(x),\ x \in \Gamma 
\ee
and
\bel{ca4} 
\pd_\nu a_1(x)=\pd_\nu a_2(x),\ x \in K \cap G.
\ee
Then the condition  $\Sigma_{a_1}=\Sigma_{a_2}$ implies  $a_1=a_2$.
\end{follow}

For our second application we consider the recovery of potentials that are known in the neighborhood of the boundary outside a compact set. In the spirit of \cite{AU}, we can improve Theorem \ref{t1} in a quite important way in that case. More precisely, we fix $R>0$ and we consider $\gamma_1$ an arbitrary open subsets of $K'\times (-\infty,-R)$, $\gamma_2$ an open subsets of $\partial\omega\times (-\infty,-R)$, $\gamma_1'$ an open subset of $K'\times (R,+\infty)$ and $\gamma_2'$ an open subsets of $\partial\omega\times (R,+\infty)$. Then, we consider the partial DN map given by
$$\Lambda_{q,R}^* : \{ h\in \mathcal H(\Gamma):\ \textrm{supp}(h)\subset (K'\times[-R,R])\cup \gamma_1\cup \gamma_1'\}\ni f \mapsto \mathcal T_1 v_{| (\partial\omega\times[-R,R])\cup \gamma_2\cup \gamma_2'}.$$
Our second application can be stated as follows

\begin{follow}
\label{c2} 
Let $\omega$ be connected, $R>0$, $\delta\in(0,R)$,  $q_1,q_2\in L^\infty(\Omega)$ be such that $q_1-q_2\in L^1(\Omega)$ and $0$ is not in the spectrum of $-\Delta+q_j$, $j=1,2$, with Dirichlet boundary condition on $\Omega$.
We fix $\omega_{1,*}$, $\omega_{2,*}$ two arbitrary $\mathcal C^2$ open and connected subset of $\omega$ satisfying $\partial\omega\subset (\partial \omega_{1,*}\cap \partial \omega_{2,*})$. We consider also $\Omega_{j,*}$, $j=1,2$, two $\mathcal C^2$ open and connected subset of $\Omega$ such that
$$\omega_{1,*}\times(-\infty, -R)\subset \Omega_{1,*}\subset \omega_{1,*}\times(-\infty, \delta-R),\quad \omega_{2,*}\times(R, +\infty)\subset \Omega_{2,*}\subset \omega_{2,*}\times(R-\delta,+\infty)$$
and we assume that 
\bel{c2a} 
q_1(x)=q_2(x),\ x \in \Omega_{1,*}\cup\Omega_{2,*}.
\ee
Then the condition  $\Lambda_{q_1,R}^*=\Lambda_{q_2,R}^*$ implies  $q_1=q_2$.
\end{follow}

In our third application we consider the recovery of potentials, supported in an infinite cylinder, appearing in a stationary Schr\"odinger equation on a slab. More precisely, for $L>0$, we consider the set $\mathcal O:=\{x=(x_1,x_2,x_3)\in \R^3:\ x_1\in (0,L)\}$, then assuming that $q\in L^\infty(\mathcal O)$ and that $0$ is not in the spectrum of $-\Delta+q$ with Dirichlet boundary condition on $\mathcal O$, we consider the problem
\bel{eqs}
\left\{ 
\begin{array}{rcll} 
(-\Delta + q) v & = & 0, & \mbox{in}\ \mathcal O,\\ 
v_{|x_1=0} & = & 0,& \\
v_{|x_1=L} & = & f .& 
\end{array}
\right.
\ee

Fixing  $r>0$, $\partial\mathcal O_{+}:=\{(L,x_2,x_3):\ x_2,x_3\in \R\}$ and $\partial\mathcal O_{-,r}:=\{(0,x_2,x_3):\ x_2\in(-r,r),\ x_3\in \R\}$, we associate to this problem the partial DN map
$$\mathcal N_{q,r}:H^{\frac{1}{2}}(\partial\mathcal O_{+})\ni f\mapsto\partial_{x_1}v_{|\partial\mathcal O_{-,r}}$$
Then, we prove the following result.
\begin{follow}
\label{c3}
Let $q_1,q_2\in L^\infty(\mathcal O)$ be such that $q_1-q_2\in L^1(\mathcal O)$ and $0$ is not in the spectrum of $-\Delta+q_j$, $j=1,2$, with Dirichlet boundary condition on $\mathcal O$. Moreover, assume that there exists $r\in(0,+\infty)$ such that
\bel{c3a}q_1(x_1,x_2,x_3)=q_2(x_1,x_2,x_3)=0,\quad (x_1,x_2,x_3)\in\{(y_1,y_2,y_3)\in \mathcal O:\ |y_2|\geq r\} .\ee
 Then, for any $R>r$,  the condition 
\bel{c3b} 
\mathcal N_{q_1,R}=\mathcal N_{q_2,R}
\ee
implies $q_1=q_2$.
\end{follow}

\subsection{Comments about the main result and the applications}

To our best knowledge this paper is the first paper proving recovery of coefficients that are neither compactly supported nor periodic for elliptic equations in unbounded domains from boundary measurements. Indeed, beside the present paper it seems that only these two cases have been addressed so far (see \cite{CKS2,CKS3,KLU,LU}). 

Like several other papers, the main tools in our analysis are suitable solutions of the equation also called complex geometric optics (CGO in short) solutions combined with Carleman estimates. It has been proved by \cite{CKS2,CKS3,KLU,LU} that for compactly supported or periodic coefficients one can apply unique continuation or Floquet decomposition in order to transform the problem on an unbounded domain into a problem on a bounded domain. Then, one can use the  CGO solutions for the problem on the bounded domain in order to prove the recovery of the coefficients under consideration. For more general class of coefficients, one can not apply such arguments and the construction of CGO solutions for the problem on an unbounded domain seems unavoidable. In this paper, using a suitable localization in space, that propagates along the infinite direction of the unbounded cylindrical domain, we introduce, for what seems to be the first time, CGO solutions that can be directly applied to the inverse problem on the unbounded domain. This makes a difference with previous related works and it allows also to derive results like Corollary \ref{c2} where the recovery of non compactly supported coefficients is proved by mean of measurements on a bounded subset of the unbounded boundary. The construction of the CGO solutions in consideration requires also some extension of arguments, like Carleman estimate and construction of the decaying remainder term, to unbounded domain that we prove  in Section 2, 3 and 4.

Let us mention that the arguments used for the construction of the CGO solutions work only if the unbounded domain has one infinite direction (or a cylindrical shape). This approach fails if the unbounded domain has more than one infinite direction like the slab. However, following the approach of \cite{KLU,LU}, by mean of unique continuation properties we prove in Corollary \ref{c3} the recovery of coefficients supported in an unbounded cylinder. Here the cylinder can be arbitrary and this result extend the one of \cite{KLU,LU} to non compactly supported coefficients. Note also that, combining the density results stated in Lemma \ref{l5}, used for the proof of Corollary \ref{c2}, with Corollary \ref{c3}, one can check that the data used by \cite{KLU,LU} for the recovery of compactly supported coefficients allow to recover more general class of coefficients supported in infinite cylinder and known on the neighborhood of the boundary outside a compact set.

In the main result of this paper, stated in Theorem \ref{t1}, we show that the partial DN map $\Lambda_q$ allows to recover coefficients $q$ which are equivalent modulo integrable functions to a fixed bounded function. This last condition is not fulfilled by the class of potential, periodic along the axis of the cylindrical domain, considered by \cite{CKS2,CKS3}. However, combining Theorem \ref{t1} with \cite{CKS2,CKS3}, one can conclude that the partial DN map $\Lambda_q$ allows  to recover the class of coefficients $q$ considered in the present paper as well as potentials $q$ which are periodic along the axis of $\Omega$.

Let us remark that in a similar way to \cite{KLU,LU}, with suitable choice of  admissible coefficients $q$, it is possible to formulate  \eqref{eqs} with $q$ replaced by $q-k^2$ and $k^2$ taking some suitable value in the absolute continuous spectrum of the operator $-\Delta+q$ with Dirichlet boundary condition. In this context, \eqref{eqs} admits a unique solution satisfying the Sommerfeld radiation condition on the infinite directions of the domain. Assuming that $q$ is chosen in such a way that these conditions are fulfilled for \eqref{eq1} and \eqref{eqs}, one can adapt the argument of the present paper to this problem.
In this paper we do not consider such extension of our main result which requires a study of the forward problem.  

Let us also observe that like in \cite{KLU,LU}, Corollary \ref{c3} can be formulated with different kinds of measurements on the side $x_1=0$ and $x_1=L$ of $\partial\mathcal O$.

\subsection{Outline}

This paper is organized as follows. In Section 2, we start by considering the CGO solutions, without boundary conditions, for the problem in an unbounded cylindrical domain. For the construction of these solutions we combine different arguments such as localization of the CGO solutions along the axis of the waveguide and some arguments of separation of variables. Then, in the spirit of \cite{BU}, we introduce in Section 3 a Carleman estimate with linear weight stated in an infinite cylindrical domain. Using this Carleman estimate, we build in Section 4 CGO solutions vanishing on some parts of the boundary. In Section 5, we combine all these results in order to prove Theorem \ref{t1}. Finally, Section 6 is devoted to the applications of the main result stated in Corollary \ref{c1}, \ref{c2} and \ref{c3}. 

\section{CGO solutions without conditions}
\label{sec2}
In this section we introduce the first class of CGO solutions of our problem without boundary conditions. These CGO solutions correspond to some specific solutions $u\in H^2(\Omega)$ of
$-\Delta u+qu=0$ in $\Omega$ for $q\in L^\infty(\Omega)$. More precisely, we start by fixing $\theta\in\mathbb S^{1}:=\{y\in\R^2:\ |y|=1\}$, $\xi'\in\theta^\bot:=\{y\in\R^2:\ y\cdot\theta=0\}$, $\xi:=(\xi',\xi_3)\in \R^3$ with $\xi_3\neq0$. Then, we consider $\eta\in\mathbb S^2:=\{y\in\R^3:\ |y|=1\}$ defined by
$$\eta=\frac{(\xi',-\frac{|\xi'|^2}{\xi_3})}{\sqrt{|\xi'|^2+\frac{|\xi'|^4}{\xi_3^2}}}.$$ 
In particular, we have
\bel{orth}\eta\cdot\xi=(\theta,0)\cdot\xi=(\theta,0)\cdot\eta=0.\ee
Then, we fix $\chi\in\mathcal C^\infty_0(\R;[0,1])$ such that $\chi=1$ on a neighborhood of $0$ in $\R$ and, for $\rho>1$, we consider  solutions $u\in H^2(\Omega)$ of
$-\Delta u+qu=0$ in $\Omega$ taking the form
\bel{CGO1}u(x',x_3)=e^{-\rho \theta\cdot x'}\left(e^{i\rho \eta\cdot x}\chi\left(\rho^{-\frac{1}{4}}x_3\right)e^{-i\xi\cdot x}+w_\rho(x)\right),\quad x=(x',x_3)\in\Omega.\ee
Here the remainder term $w_\rho\in H^2(\Omega)$ satisfies the decay property 
\bel{CGO2} \rho^{-1}\norm{w_\rho}_{H^2(\Omega)}+\rho\norm{w_\rho}_{L^2(\Omega)}\leq C\rho^{\frac{7}{8}},\ee
with $C$ independent of $\rho$. This construction can be summarized in the following way.
\begin{theorem}\label{t4} There exists $\rho_0>1$ such that, for all $\rho>\rho_0$, the equation  $-\Delta u+qu=0$ admits a solution $u\in H^2(\Omega)$ of the form \eqref{CGO1} with $w_\rho$ satisfying the decay property \eqref{CGO2}.\end{theorem}

\begin{remark}\label{r1} \emph{Comparing to CGO solutions on bounded domains, the main difficulty in the construction of  CGO solutions in our context comes from the fact that $\Omega$ is not bounded and the CGO solutions should lye in $L^2(\Omega)$. This means that the usual principal parts of the CGO solutions  considered by \cite{BU,KSU,SU}, taking the form $e^{-\rho \theta\cdot x'}e^{i\rho \eta\cdot x}e^{-i\xi\cdot x}$ in our context, will be inadequate since it will not be lying in $L^2(\Omega)$. This is the main reason why we introduce the new expression involving the cut-off $\chi$ that allows to localize such expressions. The main difficulty in our choice consist of using this expression to localize without loosing the decay properties stated in \eqref{CGO2}. This will be done by  assuming that the principal part of the  CGO solutions given by
$$e^{-\rho \theta\cdot x'}e^{i\rho \eta\cdot x}\chi\left(\rho^{-\frac{1}{4}}x_3\right)e^{-i\xi\cdot x}$$
 propagates in some suitable way along the axis of the waveguide with respect to the large parameter $\rho$. Actually, this seems to be one of the main novelty in our construction of CGO solutions comparing to any others.}\end{remark}

Clearly, $u$ solves $-\Delta u+qu=0$ if and only if $w_\rho$ solves
\begin{equation}\label{eqGO1} P_{-\rho}w_\rho=-qw_\rho-e^{\rho \theta\cdot x'}(-\Delta+q)e^{-\rho \theta\cdot x'}e^{i\rho \eta\cdot x}\chi\left(\rho^{-\frac{1}{4}}x_3\right)e^{-i\xi\cdot x},\ee
with $P_s$, $s\in\R$, the differential operator defined by
\bel{Ps}P_s:=-\Delta-2s\theta\cdot\nabla'-s^2,\ee
where $\nabla'=(\partial_{x_1},\partial_{x_2})^T$. In order to define a suitable set of solutions of \eqref{eqGO1}, we start by considering
the following equation
\begin{equation}\label{eqGO3} P_{-\rho}y=F,\quad x\in\Omega .\ee
Taking the Fourier transform with respect to $x_3$, denoted by $\mathcal F_{x_3}$, on both side of this identity we get
\begin{equation}\label{eqGO4} P_{k,-\rho}y_k=F_k,\quad k\in\R,\ee
with $F_k(x')=\mathcal F_{x_3}F(x',k)$, $y_k(x')=\mathcal F_{x_3}y(x',k)$ and $$P_{k,-\rho}=-\Delta'+2\rho\theta\cdot\nabla'-\rho^2+k^2.$$
Here $\Delta'=\partial_{x_1}^2+\partial_{x_2}^2$ and $\mathcal F_{x_3}$ is defined by
$$\mathcal F_{x_3}h(x',k):=(2\pi)^{-\frac{1}{2}}\int_\R h(x',x_3)e^{-ikx_3}dx_3,\quad h\in L^1(\Omega).$$
We fix also $p_{k,-\rho}(\zeta)=|\zeta|^2+2i\rho\theta\cdot\zeta+k^2$, $\zeta\in \R^2$, $k\in\R$, such that, for  $D_{x'}=-i\nabla'$, we have $p_{k,-\rho}(D_{x'})=P_{k,-\rho}$. Applying some  results of \cite{Ch,Ho2,I} about solutions of  PDEs with constant coefficients we obtain the following.

\begin{lemma}\label{l1} For every $\rho>1$ and $k\in\R$  there exists a bounded operator $$E_{k,\rho}:\ L^2(\omega)\to L^2(\omega)$$ such that:
\begin{equation}\label{l1a}P_{k,-\rho} E_{k,\rho}F=F,\quad F\in L^2(\omega),\end{equation}
\begin{equation}\label{l1b} \norm{E_{k,\rho}}_{\mathcal B(L^2(\omega))}\leq C\rho^{-1},\end{equation}
\begin{equation}\label{l1c} E_{k,\rho}\in \mathcal B(L^2(\omega);H^2(\omega))\ee and
\bel{l1d}\quad \norm{E_{k,\rho}}_{\mathcal B(L^2(\omega);H^2(\omega))}+\norm{k^2E_{k,\rho}}_{\mathcal B(L^2(\omega))}\leq C\rho,\end{equation}
with $C>0$ depending only on  $\omega$.\end{lemma}
\begin{proof} In light of \cite[Thoerem 2.3]{Ch} (see also \cite[Theorem 10.3.7]{Ho2}), there exists a bounded operator $E_{k,\rho}\in \mathcal B( L^2(\omega))$, defined from  fundamental solutions associated to $P_{k,-\rho}$ (see Section 10.3 of \cite{Ho2}),  such that \eqref{l1a} is fulfilled. In addition, fixing 
\[\tilde{p}_{k,-\rho}(\zeta):=\left(\sum_{\alpha\in\mathbb N^2}|\partial^\alpha_\zeta p_{k,-\rho}(\zeta)|^2\right)^{{1\over2}},\quad \zeta\in\R^2,\]
for all differential operator  $Q(D_{x'})$ with ${Q(\zeta)\over \tilde {p}_{k,-\rho}(\zeta)}$ a bounded function, we have $Q(D_{x'})E_{k,\rho}\in\mathcal B(L^2(\omega))$ and there exists  a constant $C$ depending only on $\omega$ such that
\begin{equation}\label{l1e}\norm{Q(D_{x'})E_{k,\rho}}_{\mathcal B(L^2((0,T)\times(-R,R)))}\leq C\sup_{\zeta\in\R^{2}}{|Q(\zeta)|\over \tilde {p}_{k,-\rho}(\zeta)}.\end{equation}
Note that $$\tilde{p}_{k,-\rho}(\zeta)\geq \sqrt{\abs{\im \partial_{\zeta_1} p_{k,-\rho}(\mu,\eta)}^2+\abs{\im \partial_{\zeta_2} p_{k,-\rho}(\mu,\eta)}^2}=2\rho,\quad \zeta\in\R^2.$$
Therefore, \eqref{l1e} implies
\[\norm{E_{k,\rho}}_{\mathcal B(L^2(\omega))}\leq C\sup_{\zeta\in\R^{2}}{1\over \tilde {p}_{k,-\rho}(\zeta)}\leq C\rho^{-1}\]
and \eqref{l1b} is fulfilled. In a same way, for all $\zeta\in\R^2$,  assuming that $k^2+|\zeta|^2\geq 2\rho^2$, we have
$$\tilde{p}_{k,-\rho}(\zeta)\geq \abs{\re p_{k,-\rho}(\zeta)}=k^2+|\zeta|^2-\rho^2\geq \frac{k^2+|\zeta|^2}{2}.$$
Thus, we have 
$$\sup_{\zeta\in\R^{2}}{|\zeta|^2+k^2\over \tilde {p}_{k,-\rho}(\zeta)}\leq \sup_{k^2+|\zeta|^2\geq 2\rho^2}{|\zeta|^2+k^2\over \tilde {p}_{k,-\rho}(\zeta)}+\sup_{k^2+|\zeta|^2\leq 2\rho^2}{|\zeta|^2+k^2\over \tilde {p}_{k,-\rho}(\zeta)}\leq 2+2\rho^2\sup_{\zeta\in\R^{2}}{1\over \tilde {p}_{k,-\rho}(\zeta)}\leq 3\rho.$$
Then, in view of  \cite[Theorem 2.3]{Ch}, we deduce \eqref{l1c} with
\[\norm{E_{k,-\rho}}_{\mathcal B(L^2(\omega);H^2(\omega))}+\norm{k^2E_{k,-\rho}}_{\mathcal B(L^2(\omega))}\leq C\sup_{\zeta\in\R^{2}}{1+|\zeta|^2+k^2\over \tilde {p}_{k,-\rho}(\zeta)}\leq C\rho\]
which implies \eqref{l1d}.\end{proof}
Applying this lemma, we can now consider solutions of \eqref{eqGO3} given by the following result.
\begin{lemma}\label{l2} For every $\rho>1$   there exists a bounded operator $$E_{\rho}:\ L^2(\Omega)\to L^2(\Omega)$$ such that:
\begin{equation}\label{l2a}P_{-\rho} E_{\rho}F=F,\quad F\in L^2(\Omega),\end{equation}
\begin{equation}\label{l2b} \norm{E_{\rho}}_{\mathcal B(L^2(\Omega))}\leq C\rho^{-1},\end{equation}
\begin{equation}\label{l2c} E_{\rho}\in \mathcal B(L^2(\Omega);H^2(\Omega))\ee and
\bel{l2d}\quad \norm{E_{\rho}}_{\mathcal B(L^2(\Omega);H^2(\Omega))}\leq C\rho,\end{equation}
with $C>0$ depending only on  $\Omega$.\end{lemma}
\begin{proof}
According to Lemma \ref{l1}, we can define $E_\rho$ on $L^2(\Omega)$  by
\[E_{\rho}F:=\Omega\ni(x',x_3)\mapsto\mathcal F_k^{-1}\left(E_{k,\rho}\mathcal F_{x_3} F(\cdot,k)\right)(x',x_3).\]
It is clear that \eqref{l1a} implies \eqref{l2a}. Moreover, we have
\[\norm{E_{\rho}F}^2_{L^2(\Omega)}=\int_{\R}\norm{E_{k,\rho}\mathcal F_{x_3} F(\cdot,k)}^2_{L^2(\omega)}dk\]
and  from \eqref{l1b} we get 
\[\norm{E_{\rho}F}^2_{L^2(\Omega)}\leq C^2\rho^{-2}\int_{\R}\norm{\mathcal F_{x_3} F(\cdot,k)}^2_{L^2(\omega)}dk=C^2\rho^{-2}\norm{F}^2_{L^2(\Omega)}.\]
From this estimate we deduce \eqref{l2b}. In view of \eqref{l1c}-\eqref{l1d}, we have $E_{\rho}\in \mathcal B(L^2(\Omega);H^2(\Omega))$ and, for all $F\in L^2(\Omega)$,  we have
$$\begin{aligned}\norm{E_{\rho}F}_{ H^2(\Omega)}^2&\leq C'\int_\R\left[\norm{E_{k,\rho}\mathcal F_{x_3} F(\cdot,k)}^2_{H^2(\omega)}+\norm{k^2E_{k,\rho}\mathcal F_{x_3} F(\cdot,k)}^2_{L^2(\omega)}\right]dk\\
\ &\leq C'C^2\rho^2\int_{\R}\norm{\mathcal F_{x_3} F(\cdot,k)}^2_{L^2(\omega)}dk=C'C^2\rho^2\norm{F}^2_{L^2(\Omega)},\end{aligned}$$
with $C'$ depending only on $\omega$.
This proves \eqref{l2c}-\eqref{l2d}.\end{proof}

Using this last result, we can build  geometric optics solutions of the form \eqref{CGO1}.\\
 \ \\
\textbf{Proof of Theorem \ref{t4}.} We start by recalling that
\bel{t4b}\begin{aligned}&-e^{\rho \theta\cdot x'}(-\Delta+q)e^{-\rho \theta\cdot x'}e^{i\rho \eta\cdot x}\chi\left(\rho^{-\frac{1}{4}}x_3\right)e^{-i\xi\cdot x}\\
&=-\left((|\xi|^2+q)\chi\left(\rho^{-\frac{1}{4}}x_3\right)-2i\eta_3\rho^{\frac{3}{4}}\chi'\left(\rho^{-\frac{1}{4}}x_3\right)+2i\xi_3\rho^{-\frac{1}{4}}\chi'\left(\rho^{-\frac{1}{4}}x_3\right)-\rho^{-\frac{1}{2}}\chi''\left(\rho^{-\frac{1}{4}}x_3\right)\right)e^{i\rho \eta\cdot x}e^{-i\xi\cdot x}\end{aligned}\ee
On the other hand, we have
$$\int_\R\abs{\chi\left(\rho^{-\frac{1}{4}}x_3\right)}^2dx_3=\rho^{\frac{1}{4}}\int_\R\abs{\chi(t)}^2dt$$
and we deduce that 
$$\norm{\chi\left(\rho^{-\frac{1}{4}}x_3\right)}_{L^2(\Omega)}= \norm{\chi}_{L^2(\R)}|\omega|^{\frac{1}{2}}\rho^{\frac{1}{8}}.$$
In the same way, one can check that
$$\norm{\chi\left(\rho^{-\frac{1}{4}}x_3\right)}_{L^2(\Omega)}+\norm{\chi'\left(\rho^{-\frac{1}{4}}x_3\right)}_{L^2(\Omega)}+\norm{\chi''\left(\rho^{-\frac{1}{4}}x_3\right)}_{L^2(\Omega)}\leq C\rho^{\frac{1}{8}},$$
with $C$ depending only on $\omega$ and $\chi$. Combining this with \eqref{t4b}, we find
\bel{t4c}\begin{aligned}&\norm{-e^{\rho \theta\cdot x'}(-\Delta+q)e^{-\rho \theta\cdot x'}e^{i\rho \eta\cdot x}\chi\left(\rho^{-\frac{1}{4}}x_3\right)e^{-i\xi\cdot x}}_{L^2(\Omega)}\\
&=C\left((|\xi|^2+\norm{q}_{L^\infty(\Omega)})\rho^{\frac{1}{8}}+2|\eta_3|\rho^{\frac{7}{8}}+2|\xi_3|\rho^{-\frac{1}{8}}+\rho^{-\frac{3}{8}}\right)\leq C\rho^{\frac{7}{8}},\end{aligned}\ee
with $C>0$ depending on $\omega$, $\xi$ and $\norm{q}_{L^\infty(\Omega)}$.
According to Lemma \ref{l2}, we can rewrite equation \eqref{eqGO1} as $$w_\rho=-E_{\rho}\left(qw_\rho+e^{\rho \theta\cdot x'}(-\Delta+q)e^{-\rho \theta\cdot x'}e^{i\rho \eta\cdot x}\chi\left(\rho^{-\frac{1}{4}}x_3\right)e^{-i\xi\cdot x}\right),$$ with $E_{\rho}\in\mathcal B(L^2(\Omega))$ given by Lemma \ref{l2}.
For this purpose, we will use a standard fixed point argument associated to the map
\begin{align*}
\mathcal G:  L^2(\Omega) & \to L^2(\Omega), 
\\ 
G &\mapsto -E_{\rho}\left[qG+e^{\rho \theta\cdot x'}e^{-i\rho \eta\cdot x}(-\Delta+q)e^{-\rho \theta\cdot x'}e^{i\rho \eta\cdot x}\chi\left(\rho^{-\frac{1}{4}}x_3\right)e^{-i\xi\cdot x}\right].
\end{align*}
Indeed, in view of \eqref{l2b} and \eqref{t4c}, we have
$$\norm{\mathcal Gw}_{L^2(\Omega)}\leq C\rho^{-\frac{1}{8}}+C\rho^{-1}\norm{w}_{L^2(\Omega)},\quad w\in L^2(\Omega),$$
$$\norm{\mathcal Gw_1-\mathcal Gw_2}_{L^2(\Omega)}\leq \norm{E_{\rho}[q(w_1-w_2)]}_{L^2(\Omega)}\leq C\rho^{-1}\norm{w_1-w_2}_{L^2(\Omega)},\quad w\in L^2(\Omega),$$
with $C$ depending on $\omega$, $\xi$ and $\norm{q}_{L^\infty(\Omega)}$.
Therefore, fixing $M_1>0$, there  exists $\rho_0>1$ such that for $\rho\geq \rho_0$ the map $\mathcal G$ admits a unique fixed point $w_\rho$ in $\{w\in L^2(\Omega): \norm{w}_{L^2(\Omega)}\leq M_1\}$. In addition, condition \eqref{l2b}-\eqref{l2d} imply that $w_\rho\in H^2(\Omega)$ fulfills \eqref{CGO2}. This completes the proof of Theorem \ref{t4}. \qed

\section{Carleman estimate}
\label{sec3}
In this section we derive a Carleman estimate for the Laplace operator in the unbounded cylindrical domain $\Omega$. We  consider first a  Carleman inequality similar to \cite[Lemma 2.1]{BU} for unbounded cylindrical domains.

\begin{proposition}
\label{p2} 
Let $\theta \in \mathbb S^1$. Then, there exists $d>0$ depending only on $\omega$ such that  the estimate
\bea
& & \frac{8\rho^2}{d} \| e^{-\rho \theta \cdot x'} u \|_{L^2(\Omega)}^2 + 2 \rho \| e^{-\rho \theta \cdot x'} (\theta\cdot \nu)^{1 \slash 2} \pd_\nu u \|_{L^2(\partial\omega_\theta^+\times\R)}^2
\nonumber \\
& \leq & \| e^{-\rho \theta \cdot x'} \Delta u \|_{L^2(\Omega)}^2 + 2 \rho  \| e^{-\rho \theta \cdot x'} | \theta\cdot \nu|^{1 \slash 2} \pd_\nu u \|_{L^2(\partial\omega_\theta^-\times\R)}^2,
\label{p2a}
\eea
holds for every $u \in H^2(\Omega)$ satisfying $u_{\vert \Gamma}=0$.
\end{proposition} 
\begin{proof} 
We start by proving \eqref{p2a} for $u\in\mathcal C^\infty_0(\R^3)$ satisfying $u_{\vert \Gamma}=0$.  The operator $e^{-\rho \theta \cdot x'}\Delta  e^{\rho \theta \cdot x'}$ decomposes into the sum  $  P_+' +P_+^3+ P_-$, with
$$  P_+' :=\Delta' + \rho^2\ \mbox{and}\ P_- :=2 \rho \theta \cdot \nabla',\ P_+^3 :=\pd_{x_3}^2, $$
where the symbol $\Delta'$ (resp., $\nabla'$) stands for the Laplace (resp., gradient) operator with respect to $x' \in \omega$. Thus, we get upon setting $v(x):=e^{-\rho \theta \cdot x'} u(x)$ that 
\beas
\| e^{-\rho \theta \cdot x'} \Delta u \|_{L^2(\Omega)}^2 & = & \| e^{-\rho \theta \cdot x'} \Delta e^{\rho \theta \cdot x'} v \|_{L^2(\Omega)}^2 \\
& = & \| (  P_+'+P^3_+ + P_-) v \|_{L^2(\Omega)}^2 \\
& = & \| ( P_+'+P^3_+) v \|_{L^2(\Omega)}^2 + \| P_- v \|_{L^2(\Omega)}^2 + 2 \re \langle P_+^3 v , P_- v \rangle_{L^2(\Omega)} + 2 \re \langle P_+' v , P_- v \rangle_{L^2(\Omega)}, 
\eeas
and hence
\bel{p2b}
\| P_- v \|_{L^2(\Omega)}^2 + 2 \re \langle P_+' v , P_- v \rangle_{L^2(\Omega)} \leq \| e^{-\rho \theta \cdot x'} \Delta u \|_{L^2(\Omega)}^2 - 2 \re \langle P_+^3 v , P_- v \rangle_{L^2(\Omega)}. 
\ee
Moreover, we find upon integrating by parts that
\bel{p2e} 
2\re \langle P_+^3 v , P_- v \rangle_{L^2(\Omega)} = -\rho \int_\R \int_\omega \nabla' \cdot (\abs{\pd_{x_3} v(x)}^2 \theta ) dx' dx_1=-\rho \int_{\Gamma} \abs{\pd_{x_3} v(x)}^2 \theta \cdot \nu(x)  d\sigma(x)=0 .
\ee
Here we used the fact that the condition  $v_{\vert \Gamma}=0$ implies $\pd_{x_3} v_{\vert \Gamma}=0$. 
Next, as the function $w:=v(\cdot,x_3) \in C^2(\overline{\omega})$ satisfies $w_{\vert \pd \omega}=0$ for a.e. $x_3 \in \R$, applying \cite[Lemma 2.1]{BU}, we deduce that there exists $d>0$ depending on $\omega$ such that
$$
\frac{8\rho^2}{d} \| w \|_{L^2(\omega)}^2 + 2\rho \int_{\pd \omega} e^{-2\rho  \theta \cdot x'} \theta\cdot \nu(x') \abs{\pd_\nu e^{\rho  \theta \cdot x'} w(x')}^2 d \sigma(x') 
\leq \| P_- w \|_{L^2(\omega)}^2 + 2 \re \langle P_+' w , P_-w \rangle_{L^2(\omega)}.
$$
It follows
\beas
& & \frac{8\rho^2}{d} \| e^{-\rho \theta \cdot x'} u(\cdot,x_3) \|_{L^2(\omega)}^2 + 2 \rho \int_{\pd \omega} e^{-2\rho \theta \cdot x'} \theta\cdot \nu(x) \abs{\pd_\nu u(\cdot,x_3)}^2d \sigma(x') \\
& \leq & \| P_- v(\cdot,x_3) \|_{L^2(\omega)}^2 + 2 \re \langle P_+' v(\cdot,x_3) , P_- v(\cdot,x_3) \rangle_{L^2(\omega)}.
\eeas
Thus,  integrating both sides of the above inequality with respect to $x_3\in \R$, we obtain
\bel{p2c}
\frac{8\rho^2}{d} \| e^{-\rho \theta \cdot x'} u \|_{L^2(\Omega)}^2 +2 \rho \int_{\Gamma} e^{-2\rho\theta\cdot x'} \theta \cdot \nu(x) \abs{\pd_\nu u(x)}^2 d \sigma(x)
\leq \| P_-v \|_{L^2(\Omega)}^2 + 2 \re \langle P_+' v , P_-v \rangle_{L^2(\Omega)}. 
\ee

Putting 
\eqref{p2b}-\eqref{p2c} together, we end up getting \eqref{p2a}. Finally, using the density of the space of restriction to $\Omega$ of function $u\in\mathcal C^\infty_0(\R^3)$ satisfying $u_{|\Gamma}=0$ in the space of function $u\in H^2(\Omega)$ satisfying $u_{|\Gamma}=0$, we deduce that \eqref{p2a} holds for $u\in H^2(\Omega)$.
\end{proof}

Using the fact that
$$  \abs{\Delta u}^2 \leq 2 \left( \abs{(-\Delta +q )u}^2+\norm{q}^2_{L^\infty(\Omega)}\abs{u}^2 \right), $$
we get
\beas
& & \left( \frac{4 \rho^2}{d} - \| q \|_{L^{\infty}(\Omega)}^2 \right) \| e^{-\rho \theta \cdot x'} u \|_{L^2(\Omega)}^2 + \rho \| e^{-\rho \theta \cdot x'} (\theta\cdot \nu)^{\frac{1}{2}} \pd_\nu u \|_{L^2(\check{\Gamma}_\theta^+)}^2
\nonumber \\
& \leq & \| e^{-\rho \theta \cdot x'} (-\Delta+q) u \|_{L^2(\Omega)}^2 + \rho  \| e^{-\rho \theta \cdot x'} |\theta \cdot \nu|^{\frac{1}{2}} \pd_\nu u \|_{L^2(\check{\Gamma}_\theta^-)}^2.
\eeas
As a consequence we obtain the following estimate.
\begin{follow}
\label{car1}
For $M>0$, let $q \in L^\infty(\Omega)$  satisfy $\| q \|_{L^\infty(\Omega)} \leq M$. Then, under the conditions of Proposition \ref{p2}, we have
\beas
& & \frac{2 \rho^2}{d} \| e^{-\rho \theta \cdot x'} u \|_{L^2(\Omega)}^2 + \rho \| e^{-\rho \theta \cdot x'} (\theta\cdot \nu)^{\frac{1}{2}} \pd_\nu u \|_{L^2(\check{\Gamma}_\theta^+)}^2
\nonumber \\
& \leq & \| e^{-\rho \theta \cdot x'} (-\Delta+q) u \|_{L^2(\Omega)}^2 + \rho  \| e^{-\rho \theta \cdot x'} |\theta\cdot \nu|^{\frac{1}{2}} \pd_\nu u \|_{L^2(\check{\Gamma}_\theta^-)}^2,
\eeas
provided $\rho \geq \rho_1:=M (d \slash 2)^{\frac{1}{2}}+1$.
\end{follow}

\section{CGO solutions vanishing on parts of the boundary}
In this section we fix $q\in L^\infty(\Omega)$. From now on, for all $y\in\mathbb S^{1}$ and all  $r>0$, we set
\[\partial\omega_{+,r,y}=\{x\in\Gamma:\ \nu(x)\cdot y>r\},\quad\partial\omega_{-,r,y}=\{x\in\Gamma:\ \nu(x)\cdot y\leq r\}.\]
 Here and in the remaining of this text we always assume, without mentioning it, that $y$ and $r$ are chosen in such way that $\partial\omega_{\pm,r,\pm y}$ contain  a non-empty relatively open subset of $\partial\omega$. Without lost of generality we  assume that there exists $0<\epsilon<1$ such that for all $\theta\in\{y\in\mathbb S^{1}:|y-\theta_0|\leq\epsilon\} $ we have $\partial\omega_{-,\epsilon,-\theta}\subset K'$.
The goal of this section is to use the Carleman estimate \eqref{p2a} in order to build  solutions $u\in H_{\Delta}(\Omega)$ to 
\begin{equation}
\label{(5.1)}
\left\{
\begin{array}{l}
-\Delta u+qu=0\ \ \textrm{in }  \Omega,
\\

u=0,\quad \ \textrm{on } \partial\omega_{+,\epsilon/2,-\theta}\times\R,
\end{array}
\right.
\end{equation}
 of the form
\bel{CGO1a}
u(x',x_3)=e^{\rho \theta\cdot x'}e^{-i\rho \eta\cdot x}\left(\chi\left(\rho^{-\frac{1}{4}}x_3\right)+z_\rho(x)\right),\quad x=(x',x_3)\in\Omega.\ee
Here $\theta\in\{y\in\mathbb S^{1}:|y-\theta_0|\leq\epsilon\} $,  $z_\rho \in e^{-\rho \theta\cdot x'}e^{i\rho \eta\cdot x}H_{\Delta}(\Omega)$ fulfills   $$z_\rho(x',x_3)=-\chi\left(\rho^{-\frac{1}{4}}x_3\right),\quad (x',x_3)\in\partial\omega_{+,\epsilon/2,-\theta}\times\R$$ and
\bel{CGO1b}
\| z_\rho \|_{L^2(\Omega)}\leq C\rho^{-\frac{1}{8}},
\ee
with $C$ depending on $K'$, $\Omega$ and any $M\geq\norm{q}_{L^\infty(\Omega)}$. Since $(\partial\omega\setminus K')\subset(\partial\omega\setminus \partial\omega_{-,\epsilon,-\theta})=\partial\omega_{+,\epsilon,-\theta}$,  it is clear that condition \eqref{(5.1)} implies supp$(\mathcal T_0 u)\subset K$  (recall that for $v\in\mathcal C^\infty_0(\overline{\Omega})$, $\mathcal T_{0}v=v_{|\Gamma}$).

The main result of this section can be stated as follows.

\begin{theorem}\label{t3} Let $q\in L^\infty(\Omega)$, $\theta\in\{y\in\mathbb S^{1}:|y-\theta_0|\leq\epsilon\} $. For all $\rho> \rho_1$,  one can find a solution $u\in H_{\Delta}(\Omega)$ of \eqref{(5.1)} taking the form \eqref{CGO1a} with $z_\rho$ satisfying \eqref{CGO1b}. Here $\rho_1$ denotes the constant introduced at the end of Corollary \ref{car1}. \end{theorem}

In order to prove existence of such solutions of \eqref{(5.1)} we need some preliminary tools and an intermediate result.
\subsection{Weighted spaces}

In this subsection we give the definition of some weighted spaces. We set $s\in\R$, we fix $\theta\in\{y\in\mathbb S^{1}:|y-\theta_0|\leq\epsilon\} $ and we denote by $\gamma$ the function defined on $\Gamma$ by
\[\gamma(x)=\abs{\theta\cdot\nu(x)},\quad x\in\Gamma.\]
We introduce  the spaces  $L_s(\Omega)$,  and for all non negative measurable function $h$ on $\Gamma$ the spaces $L_{s,h,\pm}$  defined respectively by
\[L_s(\Omega)=e^{-s\theta\cdot x'}L^2(\Omega),\quad L_{s,h,\pm}=\{f:\ e^{s\theta\cdot x'}h^{{1\over2}}(x)f\in L^2(\omega_{\pm,\theta}\times\R)\}\]
with the associated norm
\[\norm{u}_s=\left(\int_\Omega e^{2s\theta\cdot x'}\abs{u}^2dx\right)^{\frac{1}{2}},\quad u\in L_s(\Omega),\]
\[\norm{u}_{s,h,\pm}=\left(\int_{\partial\omega_{\pm,\theta}\times\R} e^{2s\theta\cdot x'}h(x)\abs{u}^2d\sigma(x')dx_3\right)^{\frac{1}{2}},\quad u\in L_{s,h,\pm}.\]

\subsection{Completion of the proof}

We set  the space
\[\mathcal D_0=\{v_{|\Omega}:\ v\in\mathcal C^2_0(\R^3), \ v_{\vert \Gamma}=0\}\]
and, in view of Proposition \ref{p2}, applying the Carleman estimate \eqref{p2a} to any $g\in \mathcal D_0$ we obtain
\begin{equation}\label{caca}\rho\norm{g}_\rho+\rho^{\frac{1}{2}}\norm{\partial_\nu g}_{\rho,\gamma,-}\leq C(\norm{(-\Delta+q)g}_\rho+\norm{\partial_\nu g}_{\rho,\rho\gamma,+}),\quad \rho\geq \rho_1.\end{equation}
We introduce also the space
\[\mathcal M=\{((-\Delta+q) v_{|\Omega},\partial_\nu v_{\vert\partial\omega_{+,\theta}\times\R}):\ v\in\mathcal D_0\}\]
and  think of $\mathcal M$ as a subspace of $L_\rho(\Omega)\times L_{\rho,\rho\gamma,+}$. Combining the Carleman estimate \eqref{caca} with a classical application of the Hahn Banach theorem (see \cite[Proposition 7.1]{KSU} and \cite[Lemma 3.2]{CKS2} for more detail) to a suitable linear form defined on $\mathcal M$, we obtain the following intermediate result.

\begin{lemma}\label{l3}  We fix $\partial\omega_{-,\theta}^*=\{x\in\partial\omega:\ \theta\cdot\nu(x)<0\}$. Given $\rho\geq \rho_1$, with $\rho_1$ the constant of Corollary \ref{car1}, and
\[v\in L_{-\rho}(\Omega),\quad v_-\in L_{-\rho,\gamma^{-1},-},\]
there exists  $y\in L_{-\rho}(\Omega)$ such that:\\
1) $-\Delta y+qy=v$ in $\Omega$,\\
2) $y_{\vert \partial\omega_{-,\theta}^*\times\R}=v_-$,\\
3) $\norm{y}_{-\rho}\leq  C\left(\rho^{-1}\norm{v}_{-\rho}+\rho^{-\frac{1}{2}}\norm{v_-}_{-\rho,\gamma^{-1},-}\right)$ with $C$ depending on $\Omega$, $M\geq\norm{q}_{L^\infty(\Omega)}$.
\end{lemma}

Armed with this lemma we are now in position to prove Theorem \ref{t3}. \\
\ \\
\textbf{Proof of Theorem \ref{t3}.} We need to consider $z_\rho$  satisfying
\begin{equation}
\label{w1}
\left\{
\begin{array}{l}z_\rho\in L^2(\Omega) \\
(-\Delta+q) (e^{\rho \theta\cdot x'}e^{-i\rho \eta\cdot x}z_\rho)=-(-\Delta+q)e^{\rho \theta\cdot x'}e^{-i\rho \eta\cdot x}\chi\left(\rho^{-\frac{1}{4}}x_3\right)\ \ \textrm{in }\Omega
\\
z_\rho=-\chi\left(\rho^{-\frac{1}{4}}x_3\right)\quad \textrm{on }\partial\omega_{+,\epsilon/2,-\theta}\times\times\R.
\end{array}
\right.\end{equation}
Let $\psi\in\mathcal C^\infty_0(\R^2)$ be such that   supp$(\psi)\cap\partial\omega\subset \{x\in\partial\omega:\ \theta\cdot\nu(x)<-\epsilon/3\}$ and $\psi=1$ on $\{x\in\partial\omega:\ \theta\cdot\nu(x)<-\epsilon/2\}=\partial\omega_{+,\epsilon/2,-\theta}$. Choose $v_-(x',x_3)=-e^{\rho \theta\cdot x'}e^{-i\rho \eta\cdot x}\chi\left(\rho^{-\frac{1}{4}}x_3\right)\psi(x')$, $x\in\partial\omega_{-,\theta}\times\R$.  Since $v_-(x)=0$ for $x\in \{x\in\Gamma:\ \theta\cdot\nu(x)\geq-\epsilon/3\}\times\R$  we have
$v_-\in L_{-\rho,\gamma^{-1},-}$. Fix also $$v(x)=-(-\Delta+q)e^{\rho \theta\cdot x'}e^{-i\rho \eta\cdot x}\chi\left(\rho^{-\frac{1}{4}}x_3\right),\quad x\in\Omega.$$ From Lemma \ref{l3}, we deduce that there exists $h\in H_\Delta(\Omega)$ such that
\[
\left\{
\begin{array}{ll}
(-\Delta+q) h=v&  \mbox{in}\; \Omega,
\\
h(x)=v_-(x),& x\in\partial\omega_{-,\theta}\times\R.
\end{array}
\right.\]
Then, for $z_\rho=e^{-\rho \theta\cdot x'}e^{i\rho \eta\cdot x} h$ condition \eqref{w1} will be fulfilled.  Repeating some arguments similar to Theorem \ref{t4}, we obtain
$$\norm{e^{-\rho \theta\cdot x'}(-\Delta+q)e^{\rho \theta\cdot x'}e^{-i\rho \eta\cdot x}\chi\left(\rho^{-\frac{1}{4}}x_3\right)}_{L^2(\Omega)}\leq C\rho^{\frac{7}{8}},$$
with $C$ depending only on $\omega$ and $M\geq \norm{q}_{L^\infty(\Omega)}$.
Combining this with  condition 3) of Lemma \ref{l3} we get
\[\begin{aligned}\norm{z_\rho}_{L^2(\Omega)}=\norm{h}_{-\rho}&\leq C\left(\rho^{-1}\norm{v}_{-\rho}+\rho^{-\frac{1}{2}}\norm{v_-}_{-\rho,\gamma^{-1},-}\right)\\
\ &\leq C\left(\rho^{-\frac{1}{8}}+\rho^{-\frac{1}{2}}\norm{\psi\gamma^{-\frac{1}{2}}}_{L^2(\partial\omega_{-,\theta})}\norm{\chi\left(\rho^{-\frac{1}{4}}\cdot\right)}_{L^2(\R)}\right)\\
\ &\leq C\left(\rho^{-\frac{1}{8}}+\rho^{-\frac{3}{8}}\norm{\psi\gamma^{-\frac{1}{2}}}_{L^2(\partial\omega_{-,\theta})}\norm{\chi}_{L^2(\R)}\right)\\
\ &\leq C\rho^{-{1\over8}}\end{aligned}\]
with $C$ depending only on $\Omega$ and $\norm{q}_{L^\infty(\Omega)}$. Therefore,  estimate \eqref{CGO1b} holds.  Using the fact that $e^{-\rho \theta\cdot x'}e^{i\rho \eta\cdot x} z_\rho=h\in H_{\Delta}(\Omega)$, we deduce that  $u$ defined by \eqref{CGO1a} is lying in $H_{\Delta}(\Omega)$ and is a solution of \eqref{(5.1)}. This completes the proof of Theorem \ref{t3}.\qed

\section{Uniqueness result}
This section is devoted to the proof of Theorem \ref{t1}.  From now on we set $q=q_2-q_1$ on $\Omega$ and  we assume  that $q=0$ on $\R^{3}\setminus \Omega$. Without lost of generality we assume that  for all $\theta\in\{y\in\mathbb S^{1}:|y-\theta_0|\leq\epsilon\} $ we have $\partial\omega_{-,\epsilon,\theta}\subset G'$ with $\epsilon>0$ introduced in the beginning of the previous section. 
Let   $\rho >\max(\rho_0,\rho_1) $ and set $\theta\in\{y\in\mathbb S^{1}:|y-\theta_0|\leq\epsilon\} $, $\xi:=(\xi',\xi_3)\in\R^3$ satisfying $\xi_3\neq0$ and $\xi'\in\theta^{\bot}$. According to Theorem \ref{t4}, we can consider $u_1\in H^2(\Omega)$ solving $-\Delta u_1+q_1u_1=0$ on $\Omega$ taking the form \eqref{CGO1} with $w_\rho$ satisfying \eqref{CGO2}. In addition, in view of Theorem \ref{t3}, we can fix $u_2\in H_{\Delta}(\Omega)$ a solution of \eqref{(5.1)}, with $q=q_2$, of the form \eqref{CGO1a} with $e^{-\rho \theta\cdot x'}e^{i\rho \eta\cdot x}z_\rho\in H_{\Delta}(\Omega)$ satisfying \eqref{CGO1b}.
Fix $w_1\in H_\Delta (\Omega)$ solving
 \bel{eq3}
\left\{
\begin{array}{ll}
-\Delta w_1 +q_1w_1=0 &\mbox{in}\ \Omega,
\\

\mathcal T_{0}w_1=\mathcal T_{0}u_2. &

\end{array}
\right.
\ee
Then, $u=w_1-u_2$ solves
  \bel{eq4}
\left\{\begin{array}{ll}
-\Delta u +q_1u=(q_2-q_1)u_2 &\mbox{in}\ \Omega,
\\
u(x)=0 & \mathrm{on}\  \partial \Omega,\\

\end{array}\right.
\ee
and since $(q_2-q_1)u_2\in L^2(\Omega)$,  in view of \cite[Lemma 2.2]{CKS}, we deduce that $u\in H^2(\Omega)$. Using the fact that $u_1\in H^2(\Omega)$, we can apply the Green formula to get

\[\begin{aligned}\int_\Omega (q_2-q_1)u_2u_1dx&=\int_\Omega u_1(-\Delta u+q_1 u)dx-\int_\Omega u(-\Delta u_1+q_1u_1)dx\\
&=-\int_{\Gamma} \partial_\nu u u_1d\sigma(x)+\int_{\Gamma} \partial_\nu u_1 ud\sigma(x).\end{aligned}\]
On the other hand, we have $u_{|\Gamma}=0$ and, combining \eqref{t1a} with the fact that supp$\mathcal T_0u_2\subset K$, we deduce that $\partial_\nu u_{|G}=0$. It follows that
\begin{equation}\label{t1b} \int_\Omega qu_2u_1dx=-\int_{\Gamma\setminus G}\partial_\nu uu_1d\sigma(x).\end{equation}
In view of \eqref{CGO2}, by interpolation, we have
$$\norm{w_\rho}_{L^2(\Gamma)}\leq C\norm{w_\rho}_{H^{\frac{9}{16}}(\Omega)}\leq C\left(\norm{w_\rho}_{L^2(\Omega)}\right)^{\frac{23}{32}}\left(\norm{w_\rho}_{H^2(\Omega)}\right)^{\frac{9}{32}}\leq  C\rho^{\frac{7}{16}}.$$
Thus, applying  the Cauchy-Schwarz inequality to the first expression on the right hand side of this formula, 
we get
\[\begin{aligned}\abs{\int_{\Sigma\setminus G}\partial_\nu uu_1d\sigma(x)}&\leq\int_\R\int_{{\partial\omega}_{+,\epsilon,\theta}}\abs{\partial_\nu ue^{-\rho x'\cdot \theta}(e^{-i\xi\cdot x}\chi\left(\rho^{-\frac{1}{2}}x_3\right)+w_\rho)} d\sigma(x')dx_3 \\
 \ &\leq C\left(\int_{{\partial\omega}_{+,\epsilon,\theta}\times \R}\abs{e^{-\rho x'\cdot \theta}\partial_\nu u}^2d\sigma(x)\right)^{\frac{1}{2}}\left(\norm{\chi\left(\rho^{-\frac{1}{2}}\cdot\right)}_{L^2(\R)}+\norm{w_\rho}_{L^2(\Gamma)}\right)\\
\ &\leq C\rho^{\frac{7}{16}}\left(\int_{{\partial\omega}_{+,\epsilon,\theta}\times \R}\abs{e^{-\rho x'\cdot \theta}\partial_\nu u}^2d\sigma(x)\right)^{\frac{1}{2}}\end{aligned}\]
for some $C$ independent of $\rho$. Here we have used both \eqref{CGO2} and the fact that $(\Gamma\setminus G)\subset {\partial\omega}_{+,\epsilon,\theta}\times\R$. 
Combining this estimate with the Carleman estimate stated in Corollary \ref{car1}, we find
\begin{eqnarray}&&\abs{\int_\Omega(q_2-q_1)u_2u_1dx}^2\cr
&&\leq C\rho^{\frac{7}{8}}\int_{{\partial\omega}_{+,\epsilon,\theta}\times \R}\abs{e^{-\rho x'\cdot \theta}\partial_\nu u}^2d\sigma(x)\cr
&&\leq \epsilon^{-1}C\rho^{\frac{7}{8}}\int_{{\partial\omega}_{+,\theta}\times \R}\abs{e^{-\rho x'\cdot \theta}\partial_\nu u}^2|\nu \cdot\theta| d\sigma(x)\cr
&&\leq \epsilon^{-1}C\rho^{-\frac{1}{8}}\left(\int_\Omega\abs{ e^{-\rho x'\cdot \theta}(-\Delta +q_1)u}^2dx\right)\cr
&&\leq \epsilon^{-1}C\rho^{-\frac{1}{8}}\left(\int_\Omega\abs{ e^{-\rho x'\cdot \theta}qu_2}^2dx\right)\cr
&&\leq \epsilon^{-1}C\rho^{-\frac{1}{8}}\left(\norm{q}_{L^\infty(\Omega)}\norm{\chi}_{L^\infty(\R)} \int_\Omega|q(x)|dx +\norm{q}_{L^\infty(\Omega)}^2\norm{z_\rho}^2_{L^2(\Omega)}\right).\end{eqnarray}
Here $C>0$ stands for some generic constant independent of $\rho$.
Applying the fact that $q\in L^1(\Omega)$, we deduce that
\begin{equation}\label{t1c}\lim_{\rho\to+\infty}\int_\Omega qu_2u_1dx=0.\end{equation}
Moreover, we have
\[\int_{\Omega}qu_1u_2d x=\int_{\R^3}\chi^2(\rho^{-\frac{1}{4}} x_3)q(x)e^{-i\xi\cdot x}dx+ \int_{\Omega}Y(x) dx+\int_{\Omega}Z(x) dx\]
with $ Y(x)=q(x) e^{-i\rho \eta\cdot x}z_\rho(x)w_\rho(x)$ and $$Z(x)=q(x)\chi\left(\rho^{-\frac{1}{4}}x_3\right)\left[z_\rho e^{-ix\cdot\xi}+w_\rho e^{-i\rho \eta\cdot x} \right].$$
Applying the decay estimate given by \eqref{CGO2} and \eqref{CGO1b}, we obtain
$$\int_{\Omega}|Y(x)| dx\leq \norm{w_\rho}_{L^2(\Omega)}\norm{z_\rho}_{L^2(\Omega)}\leq C\rho^{-\frac{1}{4}}$$
$$\begin{aligned}\int_{\Omega}|Z(x)| dx&\leq \norm{q}_{L^2(\Omega)}\norm{\chi\left(\rho^{-\frac{1}{4}}\cdot\right)}_{L^\infty(\R)}(\norm{w_\rho}_{L^2(\Omega)}+\norm{z_\rho}_{L^2(\Omega)})\\
\ &\leq C\norm{q}_{L^\infty(\Omega)}^{\frac{1}{2}}\norm{q}_{L^1(\Omega)}^{\frac{1}{2}}\norm{\chi}_{L^\infty(\R)}\rho^{-\frac{1}{8}}.\end{aligned}$$
with $C$ independent of $\rho$. Combining this with \eqref{t1c}, we deduce that 
$$\lim_{\rho\to+\infty}\int_{\R^3}\chi^2(\rho^{-\frac{1}{4}} x_3)q(x)e^{-i\xi\cdot x}dx=0.$$
On the other hand, since $q\in L^1(\R^3)$ and $\chi(0)=1$, by the Lebesgue dominate convergence theorem, we find
$$\lim_{\rho\to+\infty}\int_{\R^3}\chi^2(\rho^{-\frac{1}{4}} x_3)q(x)e^{-i\xi\cdot x}dx=\int_{\R^3}q(x)e^{-i\xi\cdot x}dx.$$
This proves that, for all $\theta\in\{y\in\mathbb S^{1}:|y-\theta_0|\leq\epsilon\} $,  all $\xi'\in\R^2$ orthogonal to $\theta$ and all $\xi_3\in\R\setminus\{0\}$, we have
\bel{t1d}\mathcal F\left[\mathcal F_{x_3}q(\cdot,\xi_3)\right](\xi')=(2\pi)^{-1}\int_{\R^{2}}\mathcal F_{x_3}q(x',\xi_3)e^{-i\xi'\cdot x'}dx'=0.\ee
Since $q\in L^1(\R^3)$, $\xi_3\mapsto \mathcal F_{x_3}q(\cdot,\xi_3)$ is continuous from $\R$ to $L^1(\R^2)$ and
$$|\mathcal F_{x_3}q(\cdot,\xi_3)|\leq (2\pi)^{-\frac{1}{2}}\int_\R|q(\cdot,x_3)|dx_3,$$
 by the Fubini and the Lebesgue dominate convergence theorem, we deduce that \eqref{t1d} holds for all   $\xi'\in\R^2$ orthogonal to $\theta$ and all $\xi_3\in\R$. Now using the fact that for any $\xi_3\in\R$, $\mathcal F_{x_3}q(\cdot,\xi_3)$ is supported on $\overline{\omega}$ which is compact, we deduce, by analyticity of $\mathcal F\left[\mathcal F_{x_3}q(\cdot,\xi_3)\right]$, that $\mathcal F_{x_3}q(\cdot,\xi_3)=0$. This proves that $q=0$ which completes the proof of Theorem \ref{t1}.

\section{Applications}

In this section we will prove the three applications of Theorem \ref{t1} stated in Corollary \ref{c1}, \ref{c2} and \ref{c3}.

\subsection{Application to the Calder\'on problem}
This subsection is devoted to the proof of Corollary \ref{c1}. Applying the Liouville transform, we deduce that for $u$ the solution to \eqref{a-eq1},  $v:=a^{\frac{1}{2}} u$ solves the following BVP
$$
\left\{
\begin{array}{rcll} 
(-\Delta + q_a ) v & = & 0, & \mbox{in}\ \Omega,\\ 
v & = & a^{\frac{1}{2}} f, & \mbox{on}\ \Gamma,
\end{array}
\right.
$$
where we recall that $q_a:=a^{-\frac{1}{2}} \Delta (a^{\frac{1}{2}})$. Moreover, one can check that
$$
\Sigma_a f=a^{\frac{1}{2}} \Lambda_{q_a} a^{\frac{1}{2}}f - a^{\frac{1}{2}} \left( \pd_\nu a^{\frac{1}{2}} \right) f,\ f \in H^{\frac{1}{2}}(\Gamma) \cap a_1^{-\frac{1}{2}}(\mathcal H_K(\Gamma)),
$$
where $\Sigma_a$ is defined by  \eqref{ca2}.
From this and \eqref{ca3}-\eqref{ca4}, it then follows for every $f \in H^{\frac{1}{2}}(\Gamma) \cap a_1^{-\frac{1}{2}}(\mathcal H_K(\Gamma))$, that
$$
\Sigma_{a_j} f = a_1^{\frac{1}{2}}\Lambda_{q_j} a_1^{\frac{1}{2}} f - a_1^{\frac{1}{2}} \left( \pd_\nu a_1^{\frac{1}{2}} \right) f_{\vert G},\  j=1,2,
$$
where, for simplicity, $q_j$ stands for $q_{a_j}$. As a consequence, the condition $\Sigma_{a_1}=\Sigma_{a_2}$ implies
$$
(\Lambda_{q_1}-\Lambda_{q_2})f= a_1^{-\frac{1}{2}}(\Sigma_{a_1}-\Sigma_{a_2})a_1^{-\frac{1}{2}} f=0 ,\  f \in a_1^{\frac{1}{2}}H^{\frac{1}{2}}(\Gamma) \cap (\mathcal H_K(\Gamma)).
$$
In particular, this proves that $\Lambda_{q_1}=\Lambda_{q_2}$. Since $a_j \in \mathcal A$, $j=1,2$, it is clear that $q_j \in  L^\infty(\Omega)$ and $q_1-q_2\in L^1(\Omega)$. Then, 
according to Theorem \ref{t1}, we have $q_1=q_2$. Fixing $y:=a_1^{\frac{1}{2}}-a_2^{\frac{1}{2}}\in H^2_{loc}(\Omega)$ we deduce that $y$ satisfies
$$
\left\{
\begin{array}{rcll} 
(-\Delta + q_1 ) y & = & -a_2^{\frac{1}{2}}(q_1-q_2)=0, & \mbox{in}\ \Omega,\\ 
y_{|K\cap G}=\partial_\nu y_{|K\cap G}=0.
\end{array}
\right.
$$
Combining this with results of unique continuation for elliptic equations (e.g. \cite[Theorem 1]{SS}) we have $y=0$ and we deduce that $a_1=a_2$. This completes the proof of Corollary \ref{c1}.

\subsection{Recovery of coefficients that are known in the neighborhood of the boundary outside a compact set}
This subsection is devoted to the proof of Corollary \ref{c2}. For this purpose we assume that  the conditions of  Corollary \ref{c2} are fulfilled. Let us also introduce the following sets of functions

$$ S_q:=\{ u\in L^2(\Omega):\ -\Delta u+qu=0,\ \textrm{supp}(\mathcal T_0u)\subset K\},\quad Q_q:=\{ u\in L^2(\Omega):\ -\Delta u+qu=0\},$$
$$ S_{q,\gamma_1,\gamma_1'}:=\{ u\in L^2(\Omega):\ -\Delta u+qu=0,\ \textrm{supp}(\mathcal T_0u)\subset (K'\times[-R,R])\cup \gamma_1\cup\gamma_1'\},$$
$$ Q_{q,\gamma_2,\gamma_2'}:=\{ u\in L^2(\Omega):\ -\Delta u+qu=0,\ \textrm{supp}(\mathcal T_0u)\subset (\partial \omega\times[-R,R])\cup \gamma_2\cup\gamma_2'\}.$$

We consider first the following result of density for these spaces.

\begin{lemma}\label{l5}  The space $Q_{q_1,\gamma_2,\gamma_2'}$ (resp. $S_{q_2,\gamma_1,\gamma_1'}$) is dense in $Q_{q_1}$ (resp. $S_{q_2}$) for the topology induced by $L^2(\Omega\setminus(\Omega_{1,*}\cup\Omega_{2,*}))$.\end{lemma}
\begin{proof} Due to the similarity of these two results, we consider only the proof of the density of $Q_{q_1,\gamma_2,\gamma_2'}$ in $Q_{q_1}$. For this purpose, assume the contrary. Then, there exist $g\in L^2(\Omega\setminus(\Omega_{1,*}\cup\Omega_{2,*}))$ and $v_0\in Q_{q_1}$ such that
\bel{l5a}\int_{\Omega\setminus(\Omega_{1,*}\cup\Omega_{2,*})} gvdx=0,\quad v\in Q_{q_1,\gamma_2,\gamma_2'},\ee
\bel{l5b}\int_{\Omega\setminus(\Omega_{1,*}\cup\Omega_{2,*})} gv_0dx=1.\ee
From now on, we extend g by $0$ to $\Omega$. Let $y\in H^2(\Omega)$ be the solution of
$$
\left\{
\begin{array}{rcll} 
(-\Delta + q_1 ) y & = & g, & \mbox{in}\ \Omega,\\ 
y & = & 0, & \mbox{on}\ \Gamma.
\end{array}
\right.
$$
Then, for any $v\in H^2(\Omega)\cap Q_{q_1,\gamma_2,\gamma_2'}$, we find
$$0=\int_\Omega g vdx=-\int_{\partial\omega\times[-R,R]}\partial_\nu y vd\sigma(x)-\int_{\gamma_2}\partial_\nu y vd\sigma(x)-\int_{\gamma_2'}\partial_\nu y vd\sigma(x).$$
Allowing $v\in H^2(\Omega)\cap Q_{q_1,\gamma_2,\gamma_2'}$ to be arbitrary, we deduce that 
\bel{l5c}\partial_\nu y(x)=0,\quad x\in (\partial\omega\times[-R,R])\cup \gamma_2\cup\gamma_2'.\ee
Therefore, $y$ satisfies 
$$
\left\{
\begin{array}{l} 
(-\Delta + q_1 ) y  = 0\ \  \mbox{in}\ \Omega_{1,*},\\ 
y_{|\gamma_2}=\partial_\nu y_{|\gamma_2}  = 0
\end{array}
\right.
$$
and the unique continuation property for elliptic equations implies that $y_{|\Omega_{1,*}}=0$. In the same way, we can prove that $y_{|\Omega_{2,*}}=0$ and we deduce that
$$y_{|\partial \Omega_{j,*}}=\partial_\nu y_{|\partial \Omega_{j,*}}=0,\quad j=1,2.$$
Combining this with \eqref{l5c}, we obtain 
$$y(x)=\partial_\nu y(x)=0,\quad x\in \partial(\Omega\setminus (\overline{ \Omega_{1,*}}\cup \overline{ \Omega_{2,*}})).$$
Now let us recall that, repeating the arguments used in \cite[Corollary  1.2]{BU}, one can check that, for any $y\in H^2(\Omega)$ and $z\in H_{\Delta}(\Omega)$, we have
$$\begin{aligned}&\int_{\Omega\setminus(\Omega_{1,*}\cup\Omega_{2,*})}  z\Delta y  dx-\int_{\Omega\setminus(\Omega_{1,*}\cup\Omega_{2,*})}  y\Delta z  dx\\
&=\left\langle \mathcal T_0z,\partial_\nu y\right\rangle_{H^{-\frac{1}{2}}(\Gamma\setminus(\Omega_{1,*}\cup\Omega_{2,*})),H^{\frac{1}{2}}(\Gamma\setminus(\Omega_{1,*}\cup\Omega_{2,*}))}-\left\langle \mathcal T_1z,y\right\rangle_{H^{-\frac{3}{2}}(\Gamma\setminus(\Omega_{1,*}\cup\Omega_{2,*})),H^{\frac{3}{2}}(\Gamma\setminus(\Omega_{1,*}\cup\Omega_{2,*}))}.\end{aligned}$$
Therefore, applying this integration by parts formula, we get
$$\int_{\Omega\setminus(\Omega_{1,*}\cup\Omega_{2,*})} gv_0dx=\int_{\Omega\setminus(\Omega_{1,*}\cup\Omega_{2,*})} (-\Delta + q_1 ) yv_0dx=0.$$
This contradicts \eqref{l5b} and completes the proof of the lemma.\end{proof}

Armed with this lemma we are now in position to complete the proof of Corollary \ref{c2}. \\
\ \\
\textbf{Proof of the Corollary \ref{c2}.} Let $u_1\in Q_{q_1,\gamma_2,\gamma_2'}$ and $u_2\in S_{q_2,\gamma_1,\gamma_1'}$. Repeating the linearization process described in Section 5 we deduce that $\Lambda_{q_1,R}^*=\Lambda_{q_2,R}^*$ implies
$$\int_{\Omega}(q_2-q_1)u_1u_2dx=0.$$
Then, \eqref{c2a} implies
\bel{c2e} 0=\int_{\Omega}(q_2-q_1)u_1u_2dx=\int_{\Omega\setminus(\Omega_{1,*}\cup\Omega_{2,*})}(q_2-q_1)u_1u_2dx.\ee
Combining this with the density result of Lemma \ref{l5} and applying again \eqref{c2a}, we deduce that
$$\int_{\Omega}(q_2-q_1)u_1u_2dx=\int_{\Omega\setminus(\Omega_{1,*}\cup\Omega_{2,*})}(q_2-q_1)u_1u_2dx=0,\quad u_1\in Q_{q_1},\ u_2\in S_{q_2}.$$
Finally, choosing $u_1,u_2$ in a similar way to Section 5, we can deduce that $q_1=q_2$. This completes the proof of the corollary.\qed
\subsection{Recovery of non-compactly supported coefficients in a slab}

In this subsection we consider Corollary \ref{c3}. Applying the construction of CGO solutions and the Carleman estimate of the previous sections, we will prove how one can extend the result of \cite{LU} to coefficients supported on an unbounded cylinder. For this purpose, we start by fixing $\delta\in (0,R-r)$ and $\omega$ an open smooth and connected subset of $(0,L)\times\R$ such that $(0,L)\times(-r-\delta,r+\delta)\subset\omega\subset (0,L)\times(-R,R) $. Then, we fix $\Omega:=\omega\times\R$ and   we consider the set of functions
$$\mathcal V_{q}(\Omega):=\{u\in H^1(\Omega):\ -\Delta u+qu=0\ \ \textrm{in }\Omega\},$$
$$\mathcal W_{q}(\mathcal O):=\{u_{|\Omega}:\ u\in H^1(\mathcal O),\ -\Delta u+qu=0\ \ \textrm{in }\mathcal O,\ u_{|x_1=0}=0\},$$
$$\mathcal W_{q}(\Omega):=\{u\in H^1(\Omega):\ -\Delta u+qu=0\ \ \textrm{in }\Omega,\ u_{|{x_1=0}}=0\}.$$
 Following \cite[Lemma 9]{LU}, one can check the following result of density.
\begin{lemma}\label{l7} Let $q\in L^\infty(\Omega)$ be such that $0$ is not in the spectrum of $-\Delta+q$ with Dirichlet boundary condition on $\mathcal O$. Then the set $ \mathcal W_{q}(\mathcal O)$ is dense in $\mathcal W_{q}(\Omega)$ with respect to the topology of $L^2(\Omega)$.\end{lemma}
In the same way, combining Lemma \ref{l7} with the Carleman estimate of Corollary \ref{car1} and \cite[Lemma 10]{LU}, we obtain the following important estimate.
\begin{lemma}\label{l8} Let $\theta:=(\theta_1,\theta_2)\in\mathbb S^1$ be such that $\theta_1>0$ and assume that \eqref{c3a}-\eqref{c3b} are fulfilled. Then we have
\bel{l8a} \begin{aligned}\abs{\int_{\mathcal O}(q_1-q_2)v_1v_2dx}&=\abs{\int_{\Omega}(q_1-q_2)v_1v_2dx}\\
\ &\leq C\rho^{-\frac{1}{2}}(\theta_1)^{-\frac{1}{2}}\left(\int_{\Omega}\abs{e^{-\rho x'\cdot\theta}(q_1-q_2)v_2}dx\right)\left(\int_{\Gamma\cap \{x_1=L\}}\abs{e^{\rho x'\cdot\theta}v_1}d\sigma(x)\right)\end{aligned}\ee
for all $v_1\in \mathcal V_{q_1}(\Omega)$ and for all $v_2\in \mathcal W_{q_2}(\Omega)$.
\end{lemma}

Armed with these two results, we will complete the proof of Corollary \ref{c3} by choosing suitably the solutions $v_j$, $j=1,2$, of the equation $-\Delta v_j+q_jv_j=0$ in $\Omega$. \\
\ \\
\textbf{Proof of Corollary \ref{c3}.} From now on we assume that the condition \eqref{c3b} is fulfilled. Let us first start by considering the set $\tilde{\omega}:=\{x:=(x_1,x_2,x_3):\ (-x_1,x_2,x_3)\in \omega\}\cup \omega$ and let us extend $q_2$ by symmetry to $\tilde{\omega}\times\R$ by assuming that $$q_2(-x_1,x_2,x_3)=q_2(x_1,x_2,x_3),\quad (x_1,x_2,x_3)\in\tilde{\omega}\times\R.$$
Applying the results of Section 2, we can consider $u_2\in H^2(\tilde{\omega}\times\R)$ solving $-\Delta u_2+q_2u_2=0$ in $\tilde{\omega}\times\R$ and taking the form
\bel{c3c}u_2(x):=e^{\rho \theta\cdot x'}\left(e^{-i\rho \eta\cdot x}\chi\left(\rho^{-\frac{1}{4}}x_3\right)e^{-ix\cdot\xi}+w_{2,\rho}(x)\right),\quad x:=(x',x_3)\in\tilde{\omega}\times\R,\ee
with $\theta:=(\theta_1,\theta_2)\in\mathbb S^1$  such that $\theta_1>0$, $\eta,\xi\in\R^3$ chosen in a similar way to the beginning of Section 2, and $w_{2,\rho}\in H^2(\tilde{\omega}\times\R)$ satisfying
\bel{c3d} \rho^{-1}\norm{w_{2,\rho}}_{H^2(\tilde{\omega}\times\R)}+\rho\norm{w_{2,\rho}}_{L^2(\tilde{\omega}\times\R)}\leq C\rho^{\frac{7}{8}}.\ee
Then, we fix $v_2\in H^2(\Omega)$ defined by
\bel{c3e} v_2(x_1,x_2,x_3):=u_2(x_1,x_2,x_3)-u_2(-x_1,x_2,x_3),\quad (x_1,x_2,x_3)\in\Omega.\ee
It is clear that $v_2\in \mathcal W_{q_2}(\Omega)$. In the same way, we fix $v_1\in \mathcal V_{q_1}(\Omega)$
\bel{c3f}v_1(x):=e^{-\rho \theta\cdot x'}\left(e^{i\rho \eta\cdot x}\chi\left(\rho^{-\frac{1}{4}}x_3\right)+w_{1,\rho}(x)\right),\quad x:=(x',x_3)\in\Omega,\ee
with $w_{1,\rho}\in H^2(\Omega)$ satisfying
\bel{c3g} \rho^{-1}\norm{w_{1,\rho}}_{H^2(\Omega)}+\rho\norm{w_{1,\rho}}_{L^2(\Omega)}\leq C\rho^{\frac{7}{8}}.\ee
Applying \eqref{l8a}-\eqref{c3g} and the fact that $q_1-q_2\in L^\infty(\Omega)\cap L^1(\Omega)\subset L^2(\Omega)$, in a similar way to Section 5 we deduce that
$$\lim_{\rho\to+\infty}\int_{\Omega}(q_1-q_2)v_1v_2dx=0.$$
On the other hand, we have
\bel{c3h}\int_{\Omega}(q_1-q_2)v_1v_2dx
=\int_{\Omega}(q_1-q_2)\chi^2\left(\rho^{-\frac{1}{4}}x_3\right)e^{-ix\cdot\xi}dx+\int_{\Omega} X_\rho dx,\ee
where
$$\begin{aligned}X_\rho:=& (q_1-q_2)e^{-\rho \theta\cdot x'}u_2w_{1,\rho}\\
&-(q_1-q_2)e^{-2\rho \theta_1x_1}\left(e^{i\rho \eta\cdot x}\chi\left(\rho^{-\frac{1}{4}}x_3\right)+w_{1,\rho}(x)\right)\left(e^{-i\rho \eta\cdot s(x)}\chi\left(\rho^{-\frac{1}{4}}x_3\right)e^{-ix\cdot\xi}+w_{2,\rho}(s(x))\right),\end{aligned}$$
with $s(x_1,x_2,x_3)=(-x_1,x_2,x_3)$. Combining this with the decay estimates \eqref{c3d}, \eqref{c3g} and using the fact that $\theta_1x_1>0$, we deduce that
$$\lim_{\rho\to+\infty}\int_{\Omega} X_\rho dx=0.$$
Then, \eqref{c3h} and the fact that $q_1-q_2\in L^1(\Omega)$ imply that
$$\int_{\Omega}(q_1-q_2)e^{-ix\cdot\xi}dx=0$$
and following the arguments used at the end of the proof of Theorem \ref{t1} we deduce that $q_1=q_2$. \qed
\section*{Acknowledgment}
The author would like to thank Gunther Uhlmann for his suggestions and remarks.

\end{document}